\documentclass[12pt]{amsart}
\usepackage{amsmath,amsfonts,amsthm,amssymb,enumerate,hyperref}
\usepackage[margin=2.5cm]{geometry}
\usepackage[all]{xy}
\usepackage{listings}
\usepackage{color}
\usepackage{framed}
\usepackage{comment}
\title{Effective Non-vanishing of Asymptotic Adjoint Syzygies}
\author{Xin Zhou}
\address{University of Michigan,
Department of Mathematics,
530 Church St., East Hall,
Ann Arbor, MI 48109-1043,
USA
}
\email{paulxz@umich.edu}
\date{}

\renewcommand{\P}{\mathbb{P}}

\renewcommand{\O}{\mathcal{O}}
\newcommand{\I}{\mathcal{I}}
\newcommand{\ra}{\rightarrow}
\newcommand{\ol}{\overline}

\numberwithin{equation}{section} 

\newtheorem{thm}{Theorem}[section]
\newtheorem{defn}[thm]{Definition}
\newtheorem{prop}[thm]{Proposition}
\newtheorem{lem}[thm]{Lemma}

\newtheorem{rmk}[thm]{Remark}

\newtheorem*{thm*}{Theorem}
\newtheorem*{eg*}{Example}

\begin{document}

\maketitle

\section*{Introduction}

The purpose of this paper is to establish an effective non-vanishing theorem for the syzygies of an adjoint-type line bundle on a smooth variety,  as the positivity of the embedding increases. In particular, we give an answer to Problem 7.9 in \cite{EL} in this setting. 

Let $X$ be a smooth projective variety of dimension $n$ over $\mathbb{C}$, and let $L$ be a very ample line bundle on $X$. Then $L$ defines an embedding:
\[X \hookrightarrow \P^{r(L)}  = \P H^0(X,L) = \textnormal{Proj} \  S\]
where $r(L) = h^0(X,L) - 1$ and $S = \textnormal{Sym} H^0(X,L)$.
Given a divisor $B$ on $X$, write:
\[R(X,B;L) = \bigoplus_m H^0(X, mL + B)\]
which is viewed as a finitely generated graded $S$-module. We will be interested in the syzygies of $R(X,B;L)$ over $S$. Specifically, $R$ has a graded minimal free resolution
\[\mathbb{F}: ... \ra F_p \ra ... \ra F_0 \ra R \ra 0\]
where $F_p = \oplus_j S(-a_{p,j})$ is a free $S$-module.  Write $K_{p,q}(X,B;L)$ for the finite dimensional vector space of minimal $p$-th syzygies of degree $(p+q)$, so that:
\[F_p \cong \bigoplus_q K_{p,q}(X,B;L) \otimes_{\mathbb{C}} S(-p-q)\] 

Ein and Lazarsfeld \cite{EL} studied these groups for $L = L_d = dA$
when $A$ is a very ample line bundle and $d$ is very large. It is elementary that: 
\begin{itemize}
\item For $q > n+1$, $K_{p,q}(X,B;L_d) = 0$
\item For $q = 0$ or $q = n+1$, $K_{p,q}(X,B;L_d) \neq 0$ for only finitely many values of $p$, which can be determined completely.
\end{itemize}
So the interesting question is when $K_{p,q}(X,B;L_d)$ is nonzero for $1 \leq q \leq n$. The first main result of \cite{EL} was that if one fixes $1 \leq q \leq n$, and if $d \gg 0$, then 
\[K_{p,q}(X,B;L_d) \neq 0\]
 for \[O(d^{q-1}) \ \leq \ p \ \leq \ r(d) - O(d^{n-1}).\]
\begin{rmk}\textnormal{
Another result of similar flavor on surfaces is \cite[Prop. 3.4]{EGHP}. 
}
\end{rmk}
The second main result of \cite{EL} was the effective statement for $X = \P^n$, $B \in |\O_{\P^n}(k)|$, $1 \leq q \leq n$ and $d$ large. Specifically,
\[K_{p,q}(\P^n,B;L_d) \neq 0\] 
for 
\small\[{q+d \choose d} - {d-k-1 \choose q} - q \leq p \leq {d+n \choose n} - {d+n-q \choose n-q}+ {k+n \choose n-q} -q -1\]
\normalsize

Our purpose here is to show that for an adjoint type divisor $B = K_X+ bA$ with $b \geq n+1$,
one can in fact obtain an effective statement for arbitrary $X$ which specializes to the statement above on Veronese syzygies. Before stating the theorem, we fix some notations. Given very ample line bundles $H_1,...,H_c$ on $X$, let $Z$ be a complete intersection of divisors from $|H_1|, ..., |H_c|$. Write: 
\[\phi(H_1,...,H_c;L_d) = h^0(Z,L_d).\]
Via the Koszul resolution of $\O_Z$, for sufficiently large $d$, $\phi(H_1,...,H_c;L_d)$ can be expressed, independent of the choice of the particular divisors, which will be the case for us, as an alternating sum of terms of the form $h^0(X,dL-\Sigma_{j \in J} H_j)$, where $J \subseteq \{1,...,c\}$. The following two special cases appear in the statement of our main result:
\[n_d := \phi(-K_X-(n-q)A+B,\underbrace{A,...,A}_{n-q};L_d)\]
\[N_d := \phi((d-q)A-B,\underbrace{A,...,A}_{q};L_d)\]

Our main result is:
\begin{thm*}
Fix $1\leq q \leq n$. Then for sufficiently large $d$, 
\begin{equation}
K_{p,q}(X,B;L_d) \neq 0
\end{equation} for every value of $p$ satisfying:
\begin{equation}\label{range}
n_d - q \leq \ p \ \leq h^0(X,L_d) -N_d - q - 1
\end{equation}
\end{thm*}

It may be instructive to see an example of how the theorem works.  
Let $X = \P^2$, and put $B = 0$, $A = \O_X(1)$. Then we're looking at the minimal free resolution of the image of $\P^2$ in its $d$-th Veronese embedding, and we work out the statement in the case $q = 1$.
\begin{itemize}
\item In the lower bound, $Z$ is 2 points and $n_d = h^0(Z,\O_{\P^2}(d)) = 2$.
\item In the upper bound, $Z'$ consists of $d$ points and $N_d = h^0(Z, \O_{\P^2}(d)) = d$.

\end{itemize}
So for $d$-th Veronese embedding of $\P^2$ with large $d$, the theorem asserts that:
\small \[K_{p,1}(X,0;dL) \neq 0, \textnormal{for} \  \ 1 \leq p \leq {d+2 \choose 2} - d - 2\]
\normalsize
which was a result of Ottaviani and Paoletti, cf. \cite{OP}. More generally, one can recover $\linebreak$ \cite[Thm. 6.1]{EL}. (See Example \ref{eg of projective space} below.)

The proof of the theorem follows the line of attack of \cite{EL}, which involves constructing secant varieties that exhibit syzygies. The main new observation here is that for adjoint $B$, one can work with secant varieties that do not vary with $d$. This greatly simplifies the calculations, and gives an effective statement which specializes to the case of Veronese embeddings. For a number of facts we use in the proof, we will refer to \cite{EL} when appropriate, instead of repeating the arguments in detail. In \textsection 1, we give a proof of the main theorem. In \textsection 2, we work out some examples. 

I am grateful to Mihai Fulger, Bill Fulton, Thomas Lam, Linquan Ma, Mircea Mustata, Zhixian Zhu for helpful comments and discussions. I am especially indebted to Rob Lazarsfeld for introducing me to the topic, numerous suggestions and encouragements. 

\section{Proof of the main result}

We first give an alternative definition of $K_{p,q}(X,B;L)$ used in the proof. Throughout, $X$ is a smooth complex projective variety of dimension $n$ and $L$ and $A$ are very ample divisors on $X$. (We will be setting $L = dA$ later.) Let \[V = H^0(X,L)\] and write $V_X$ for $V \otimes_{\mathbb{C}} \O_X$, the trivial vector bundle on X modeled on V.

Fix an integer $b \geq n+1$ and set 
\begin{equation}\label{condition on B}
B = K_X + bA + P,
\end{equation}
where $K_X$ is a canonical divisor on X and P is globally generated (we have to include globally generated instead of having only trivial for the application of duality after Prop. 1.12). Note that the higher cohomologies of $B$ vanish thanks to Kodaira vanishing. As in \cite{G} and \cite{EL}, define $K_{p,q}(X,B;L)$ to be the cohomology at the middle, of the following complex:
\[\wedge^{p+1} V \otimes H^0(B + (q-1)L) \ra \wedge^p V \otimes H^0(B + qL) \ra \wedge^{p-1} V \otimes H^0(B + (q+1)L)\]
As motivated in the introduction, we will fix the index $q \in[1,n]$. As is well-known, these Koszul cohomology groups are governed by the cohomologies of the vector bundle $M_L$ on X defined by the exact sequence:
\begin{equation}
0 \ra M_L \ra V_X \ra L \ra 0
\end{equation}

\begin{prop}\label{prop1.1}
For $1 \leq q \leq n$, $K_{p,q}(X,B;L) = H^q(\wedge^{p+q}M_L \otimes \O_X(B))$.
\end{prop}

\begin{proof}
Recalling that $B$ has vanishing higher cohomologies, the conclusion follows as in \cite[Prop. 3.2]{EL}, and the proof of \cite[Prop. 3.3]{EL}. See also \cite[Sect 1]{EL93}, \cite[Thm. 5.8]{E}. 
\end{proof}

Next, we recall the following construction from \cite[\textsection 3]{EL}. Suppose a quotient $\pi: V \twoheadrightarrow W$ of $V$ with $\dim W = w$ and a subscheme $Z$ of $X$ satisfy:
\[Z = X \cap \P(W)\]
in $\P(V)$ scheme-theoretically. Then we have the diagram: 
\begin{equation}
\begin{gathered}
\xymatrix{
0 \ar[r] & M_L \ar[r] \ar[d] & V \otimes O_X \ar[r] \ar[d] & L \ar[r] \ar[d] & 0 \\
0 \ar[r] & \Sigma_W \ar[r] & W \otimes O_X \ar[r] & L \otimes O_Z \ar[r] & 0 \\
}
\end{gathered}
\end{equation}
whose bottom exact sequence defines $\Sigma_W$, a torsion-free sheaf on X of rank $w$. Furthermore, as in \cite[\textsection 3]{EL},  $\wedge^w\Sigma_W$ maps onto $\I_{Z/X}$ and one gets a surjective map:
\begin{equation}
\sigma_\pi : \wedge^w M_L \ra \I_{Z/X}
\end{equation}

We modify \cite[Def. 3.8]{EL}:

\begin{defn}\label{def1.2}\textnormal{
We say that $W$ carries weight $q$ syzygies for $B$ if the map induced by $\sigma_\pi$:
\[H^q(X, \wedge^wM_L(B)) \ra H^q(X, \I_{Z/X}(B)) \]
is surjective. (We also say the same for $q = 0$ for notational convenience even though it isn't necessarily directly related to syzygies. )
}
\end{defn}

\begin{rmk}\label{nonvanishing remark}
\textnormal{
If for some $q \geq 1$, $H^q(X,\I_{Z/X}(B)) \neq 0$ and $W$ carries weight $q$ syzygies for $B$, then combining Prop. \ref{prop1.1} and Def. \ref{def1.2} gives us $K_{w-q,q}(X,B;L) \neq 0$
}
\end{rmk}

The lemma below describes the same kind of inductive behavior as in \cite[Thm. 3.10]{EL}, with our new definition. Let's recall some notations first. Take a general divisor $\ol{X} \in |A|$ so that $\ol{X}$ is smooth, irreducible and so that (1.3) remains exact after tensoring with $O_{\ol{X}}$. Let
\[V' = V \cap H^0(X,I_{\ol{X}/X}(L))\]
with the intersection inside $V$. Set $W' = \pi(V')$. Write
\[\ol{V} = V/V', \ \ol{W} = W/W', \ \ol{L} = L |_{\ol{X}}, \ \ol{B} = B |_{\ol{X}}, \ \ol{Z} = Z \cap \ol{X}, \ \ol{A} = A |_{\ol{X}}, \ \ol{P} = P |_{\ol{X}}\]

As in \cite[(3.14)]{EL}, we get the analogue of (1.3) above for the barred objects and we have the surjection:
\[\ol{\sigma}: \wedge^{\ol{w}} M_{\ol{V}} \ra I_{\ol{Z}/\ol{X}}\]
so we can study the behavior of $\ol{W}$ with respect to carrying syzygies. In fact, the same kind of argument as in the proof of \cite[Thm. 3.10]{EL} yields:

\begin{lem}\label{inductive lemma}
For $q \geq 1$, if $\ol{W}$ carries weight $q-1$ syzygies for $\ol{B} + \ol{A}$ on $\ol{X}$ and if
\[H^q(X,I_{Z/X}(B+A)) = 0,\]
then $W$ carries weight q syzygies for $B$ on $X$.
\end{lem}

\begin{rmk}\textnormal{
The assumption under which we'll apply the lemma is that $X$ has dimension at least 2. (The dimension one case is completely understood, cf. \cite[Prop. 5.1]{EL}.)
For the readers familiar with \cite{EL}, $K_{p,q}$ with the same $q$ are described by cohomology one different in indices in this paper. Compare Rmk. \ref{prop1.1} and \cite[Cor. 3.3]{EL}. This explains why $K_{p,1}$ is treated differently in \cite{EL} (Compare \cite[Thm. 4.1]{EL} with Prop.  \ref{syzygies carrying for effective d}.) 
}\end{rmk}

We next give an analogue of \cite[Def. 5.3]{EL}. Take $(n+1-q)$ divisors 
\[D_1 \in |-K_X -(n-q)A + B| , \ D_2 \dots D_{n+1-q} \in |A|,\] 
and let $Z = D_1 \cap D_2 ... \cap D_{n+1-q}$.
\begin{defn}\textnormal{
We say that Z is adapted to $(X,B,L,n,q)$, if $Z$ is constructed from $X,B,A,n,q$ as above and the $\{D_i\}$ intersect transversely.  
}
\end{defn}

To get started, we take $D_i$ general in its linear series. Then we have:
\begin{prop}\label{basic adaptive properties}\textnormal{
\begin{enumerate}[(i)]
\item One has, for every  $J \subsetneq \{1,2,...n+1-q\}$ and $i > 0$, $$H^i(X, B-\Sigma_{j \in J} D_j) = H^i(X, B + A -\Sigma_{j \in J} D_j)=0.$$
\item $H^q(X, I_{Z/X}(B)) \neq 0$ and $H^q(X, I_{Z/X}(B + A)) = 0$.
\item If $\ol{Z} = Z \cap \ol{X}$, then $\ol{Z}$  is adapted to $(\ol{X},\ol{B} + \ol{A}, \ol{A}, n-1, q-1)$.
\end{enumerate}
}
\end{prop}

\begin{proof}
\begin{enumerate}[(i)]
\item Thanks to the choice of $D_i$, this follows from Kodaira vanishing. 
\item Since $X$ is a smooth projective variety and the $D_i$ meet transversely, $Z$ is a complete intersection. So $I_{Z/X}$ is resolved by the Koszul complex with $j$-th term $\wedge^jE$ where $E = \oplus_{i = 1}^{n+1-q} O_X(-D_i)$. Use the Koszul resolution twisted by $B$ and $(i)$ to find:
\[H^q(X, I_{Z/X} (B)) = H^{q+n+1-q-1}(X, B - (\Sigma_{i=1}^{n+1-q} D_i)) = H^n(X, K_X) \neq 0\]
The other claim follows similarly using Kodaira vanishing. 
\item Using adjunction, we have:
\[\ol{B} = K_{\ol{X}} + (b-1)\ol{A}+ \ol{P}\] 
so $\ol{B}$ has the shape as in \eqref{condition on B}. As $\ol{X}$ is general, and $D_i$'s are general, we can assume $\{\ol{D_i}\}$ meet transversely. (Similarly, in the finite number of steps in the induction, we are free to assume the corresponding divisors intersect transversely.) $\ol{D_1}$ is in the correct linear series by adjunction. The rest is immediate. 
\end{enumerate}
\end{proof}

Now let $L_d = dA$ and we take $d$ large enough so that for any $i$, $H^i(X, L_d-iA) = 0$ (i.e. the embedding $X \subset |A|$ is $d$-regular in the sense of Castelnuovo-Mumford).
\begin{defn}\textnormal{
We say that $d$ satisfies the effective conditions for $B$, if $L_d - nA - B$ is very ample. 
}
\end{defn}

\begin{rmk}\label{surjectivity remark}
\textnormal{Assume that $d$ satisfies the effective conditions for $B$. Note that $\Sigma_{i=1}^cD_i = B - K_X$ by construction. Via the Koszul resolution on $\{D_i\}$ twisted by $L_d$ and using Kodaira vanishing, the following statements hold and furthermore, they hold after cutting down by hyperplanes repeatedly until we reach the base case of the inductive proof in Prop. \ref{syzygies carrying for effective d}: 
\begin{enumerate} 
\item The map $H^0(X,L_d) \ra H^0(Z,L_d)$ is surjective, or equivalently, $H^1(X,I_{Z/X}(L_d)) = 0$.
\item The map $H^0(Z,L_d)  \ra H^0(\ol{Z},L_d)$ is surjective, or equivalently, $H^1(Z,L_d - A) = 0$.
\item $H^1(X,I_{Z/X}(L_d-A)) = 0$ (or equivalently, with $W'$ chosen below, the map $V' \ra W'$ is surjective.)
\item The map $H^0(X,L_d)  \ra H^0(\ol{X}, L_d)$ is surjective, or equivalently $H^1(X,L_d-A) = 0$.
\end{enumerate}
}
\end{rmk}

\begin{rmk}\textnormal{
In \cite{EL}, complete intersections $Z_d$ were chosen that varied with $d$. The surjectivity of the above four maps cannot be guaranteed. This resulted in the ineffectivity and a number of complications which we are able to circumvent as in the proof of Prop. \ref{syzygies carrying for effective d}. 
}
\end{rmk}

\begin{prop}\label{syzygies carrying for effective d}
For $1 \leq q \leq n-1$, if $d$ satisfies the effective conditions for $B$, then $H^0(Z,L_d)$ carries weight $q$ syzygies for $B$.
\end{prop}
\begin{proof}
Set $W = H^0(Z,L_d)$.  When $d$ satisfies the effective conditions for $B$, $L_d-D_i$ is globally generated for all $i$.  So $I_{Z/X} \otimes \O_X(L_d)$ is globally generated. Then $Z = X \cap \P(W)$ scheme-theoretically. 
Furthermore, by Rmk. \ref{surjectivity remark}, we have the following diagram: 
\begin{equation}
\xymatrix{
0 \ar[r] &V'= H^0(X,L_d-A) \ar[r] \ar[d] &  V = H^0(X,L_d) \ar[r] \ar[d] & \ol{V} = H^0(\ol{X},L_d) \ar[d] \ar[r] & 0\\ 
0 \ar[r] & W'=H^0(Z,L_d-A) \ar[r] \ar[r] \ar[d] & W = H^0(Z,L_d) \ar[r] \ar[d] &  \ol{W}=H^0(\ol{Z},L_d) \ar[d] \ar[r] & 0\\
& 0  & 0 & 0\\
}
\end{equation}

Moreover, when we cut down by hyperplanes as in Lemma \ref{inductive lemma}, we obtain the corresponding diagrams in lower dimensions. 

We prove the proposition by induction on $q$. Notice that by construction, $Z$ is always of dimension $q-1$. Our assumption $\dim X \geq 2$ (below Remark 1.3) will always be satisifed, because $\dim X = n \geq q+1 \geq 2$. When $q = 1, Z$ consists of points. $\ol{Z} = \phi, \ol{W} = 0$. So the conclusion is trivially true for $q = 0$. Then the conclusion is true for $q = 1$ by Prop. \ref{basic adaptive properties} (ii) and Lemma \ref{inductive lemma}. Then apply Lemma \ref{inductive lemma} repeatedly. 

\end{proof}

Recalling Rmk. \ref{nonvanishing remark}, the previous proposition gives $K_{p,q}(X,B;L_d) \neq 0$ for a specific $p$. We in fact get a range of non-vanishings by enlarging $W$ but keeping $\ol{W}$ fixed, as in \cite[Thm. 3.11]{EL}. The result is:

\begin{prop}\label{expansion} Fix $1 \leq q \leq n-1$. If $d$ satisfies the effective conditions for $B$, and \linebreak $h^0(X,L_d) - h^0(Z,L_d) > n$ then $K_{p,q}(X,B;L_d) \neq 0$
for $p$ in the range: 
\[h^0(Z,L_d) -q \ \leq \ p \ \leq \ h^0(X,L_d-A) + h^0(\ol{Z},L_d) - q \ \ \ \ \ \  \square\]
\end{prop}

Now we work to apply duality to the above proposition. In order to prove Thm. \ref{main theorem}, we combine the above proposition with the range we get using duality.

Let $B' = L_d - B + K_X$. Notice that when $d$ is large, $B'$ will be of the form 
\begin{equation}
K_X + b'A + P'
\end{equation}
with $b' \geq n+1$ and $P'$ globally generated. We work with $d$ large enough so that $B'$ is indeed of this form.  

\begin{prop}\label{duality}
For $1 \leq q \leq n$,
\[K_{p,q}(X,B;L_d) = K_{r_d-p-n,n-q}(X,B';L_d)^*\]
where $r_d = h^0(X,L_d)-1$.
\end{prop}

\begin{proof}
See \cite[Thm. 2.c.6]{G}, \cite[Prop. 3.5]{EL} for $1 \leq q \leq n-1$ and combine with \cite[Rmk. 3.4]{EL} for $q = n$. 
\end{proof}

We want to apply Prop. \ref{expansion} to $B'$, let $Z'$ be a complete intersection adapted to $(X,B',A,n,n-q)$. Denote by $D_i'$, the general divisors in the corresponding linear series. 

\begin{lem}\label{B' effective}
If $d$ satisfies the effective condition for $B$, then $d$ satisfies the effective condition for $B'$.
\end{lem}

\begin{proof}
By definition of $B'$ and \eqref{condition on B}, $L_d - nA - B' = L_d - nA - (L_d - B + K_X) = B - nA - K_X$, which is very ample. 
\end{proof}

\begin{rmk}\label{rmk 1.14}\textnormal{
Note that when $d$ is large, the other assumption in Prop. \ref{expansion} is also satisfied for $B'$. The interested reader can check this keeping in mind that $Z'$ is always contained in a divisor in $|A|$ and use surjections as those in Rmk. \ref{surjectivity remark}. 
}
\end{rmk}

Finally, we arrive at the main result:
\begin{thm}\label{main theorem}
Fix $1 \leq q \leq n$. For sufficiently large $d$, if 
\[h^0(Z,L_d) - q \leq p \leq h^0(X,L_d) - h^0(Z',L_d) - q - 1\]
then $K_{p,q}(X,B;L_d) \neq 0$
\end{thm}
\begin{proof}
For $1 \leq q \leq n-1$, by Prop. \ref{expansion}, we have $K_{p,q}(X,B;L_d) \neq 0$ for
\begin{equation}
h^0(Z,L_d) - q \leq p \leq h^0(X,L_d-A) + h^0(\ol{Z},L_d) - q
\end{equation}
By Lemma \ref{B' effective} and Remark \ref{rmk 1.14}, we can apply Prop. \ref{expansion} to $B'$ and we have nonvanishings for $p$ in range:
\[h^0(Z',L_d) - (n-q) \leq r_d-p-n \leq h^0(X,L_d - A) + h^0(\ol{Z'},L_d) - (n-q)\]
ie. 
\begin{equation}
r_d-n-(h^0(X,L_d - A) + h^0(\ol{Z'},L_d) - (n-q)) \leq p \leq r_d-n-(h^0(Z',L_d) - (n-q))
\end{equation}
Now we show that the right hand side of (1.8) is of higher order than the left hand side of (1.9), so the two ranges overlap. The right hand side of (1.8) has order $O((d-1)^n)=O(d^n)$. On the left hand side of (1.7), terms of order $d^n$ appear in $r_d$ and $h^0(X,L_d-A)$ with the same coefficient and hence cancel. Therefore, the order is bounded by $O(d^{n-1})$. Hence, asymptotically, we have nonvanishing for everything between the left hand side of (1.8) and the right hand side of (1.9).

For $q = n$, we have from Prop. \ref{duality}, and \cite[Prop. 5.1]{EL} that $K_{p,n}(X,B;dL) \neq 0$ if and only if
\[ 0 \leq r(L_d)-p-n \leq r(B').\]
This unwinds to be
\[h^0(X,L_d) - h^0(X,L_d-bA) - n \leq p \leq h^0(X,L_d) - n - 1\]
Therefore, the statement is not only true, but also sharp for $q = n$. 
\end{proof}

\section{Examples}

We conclude by working out the statement of the main theorem in some interesting special cases.

\subsection{Projective space}\label{eg of projective space}

Take $X = \P^n$. Let $B$ be a divisor in $|\O_{\P^n}(k)|$. Assume $k \geq 0$, so that $B$ satisfies \eqref{condition on B}. In this case, $Z$ is a complete intersection of $(n-q)$ divisors in $|\O_{\P^n}(1)|$ and a divisor in $|-K_X-(c-1)L+B| = |\O_{\P^n}(k+q+1)|$. So,
\small\[h^0(Z,dL) = {q+d \choose d} - {d-(k+q+1) + q \choose q}\]\normalsize
$Z'$ is a complete intersection of $q$ divisors in $|\O_{\P^n}(1)|$ and a divisor in $|-K_X-(c'-1)L + B'| = |(d-k-q)L|$. So,
\small\[h^0(Z',dL) = {d+n-q \choose n-q} - {n-q + d-(d-k-q) \choose d-(d-k-q)}\]\normalsize
\small\[ = {d+n-q \choose n-q} - {k+n \choose n-q}\]\normalsize
So by Thm 1.15, for large $d$, $K_{p,q}(X,B;dL) \neq 0$ for $p$ in the range:
\small\[{q+d \choose d} - {d-k-1 \choose q} - q \leq p \leq {d+n \choose n} - {d+n-q \choose n-q}+ {k+n \choose n-q} -q -1\]\normalsize
This corollary of Thm. \ref{main theorem} coincides with the result \cite[Thm. 6.1]{EL}, but the proof is simpler since we are able to take $Z$ independent of $d$ and is a specialization of the general result. 

\subsection{Product of projective spaces}
Let $X = \P^s \times \P^t$, so $n = \dim X = s+t$. Divisors $B$ satisfying \eqref{condition on B} are of type $(u,v)$ with $u \geq t+1, v \geq s+1$. As mentioned in the introduction, we can compute $n_d$, and $N_d$ in the statement of the theorem (cf. \eqref{range}) through the Koszul resolution. Via the Kunneth formula, we find, for $1 \leq q \leq n$, sufficiently large $d$:
\[K_{p,q}(X,B;L_d) \neq 0\]
for $p$ in range:
\small
\[\sum_{i=0}^{s+t-q}(-1)^i{s+t-q \choose i}{d-i+s \choose s}{d-i+t \choose t}\]
\[+\sum_{i=0}^{s+t-q} (-1)^{i+1}{s+t-q \choose i}{d-i-u-q+t+s \choose s}{d-i-v-q+s+t \choose t} - q\]
\[\leq \ p \ \leq\]
\[{d+s \choose s}{d+t \choose t} - \left(\sum_{i=0}^{q}(-1)^i{s+t-q \choose i}{d-i+s \choose s}{d-i+t \choose t}\right)\]
\[-\left(\sum_{i = 0}^q (-1)^{i+1}{s+t-q \choose i}{q+u-i+s \choose s}{q+v-i+t \choose t}\right) - q - 1\]
\normalsize

\begin{rmk}\textnormal{
The minimal free resolutions of classical Segre or multi-Segre embeddings with line bundles of type $(1,...,1)$ are much studied. $\P^{n_1} \times ... \times \P^{n_m}$ is studied by Rubei regarding $N_p$ properties in \cite{R} and \cite{R07}, and Snowden in \cite{S} proves a finiteness theorem as we vary the number of direct product factors and the dimensions of the projective spaces. Netay studies $\P^m \times \P^n$ in \cite{N}, giving an algorithm for computing the groups as representations.  
}
\end{rmk}

\subsection{Grassmannian Gr(2,4)}

Let $X = \textnormal{Gr(2,4)}$, 2 dimensional subspaces of $\mathbb{C}^4$. Then $\textnormal{Pic}(X) = \mathbb{Z}$ and $\O_X(K_X)= \O_X(-4)$. For the Plucker embedding, the embedding line bundle $A = O_X(1)$.  Assume $B$ is of type $\O_X(k)$ where $k \geq 1$ (satisfying \eqref{condition on B}). 
Then:
\small
\[h^0(Z,L_d) = \sum_{i=0}^{4-q} (-1)^i {4-q \choose i}h^0(X,\O_X(d-i)) + \sum_{i=0}^{4-q} (-1)^{i+1} {4-q \choose i}h^0(X,\O_X(d-i-(k+q)))\]
\[h^0(Z',L_d) =\sum_{i=0}^q (-1)^i {q \choose i} h^0(X,\O_X(d-i)) + \sum_{i=0}^q (-1)^{i+1}{q \choose i} h^0(X,\O_X(k+q-i))\]
\normalsize
Using the Borel-Weil-Bott theorem, $H^0(\O_X(m))$ corresponds to $GL_4$-representations corresponding to rectangular Young diagrams with 2 rows and $m$ columns. The dimension of these representations are given by:
\[f(m) := \frac{(m+1)(m+2)(m+2)(m+3)}{12}\]
(cf. \cite[Thm 6.3 (1)]{FH}). Combine with the expressions for $h^0(Z,L_d)$ and $h^0(Z',L_d) $, we get for $1 \leq q \leq 4$ and sufficiently large $d$, $K_{p,q}(X,B;L_d) \neq 0$ for $p$ in range:
\small
\[\sum_{i=0}^{4-q} (-1)^i {4-q \choose i}f(d-i) + \sum_{i=0}^{4-q} (-1)^{i+1} {4-q \choose i}f(d-i-k-q) - q\]
\[\leq \ p \leq \]
\[r(d) - \left(\sum_{i=0}^q (-1)^i {q \choose i} f(d-i) \right) - \left(\sum_{i=0}^q (-1)^{i+1}{q \choose i} f(k+q-i)\right) - q -1 \]
\normalsize
The same can be done in principle for any Grassmannian as requested in \cite[Problem 7.9]{EL}.

\end{document}